\documentclass[11pt]{amsart}
\usepackage{geometry}                
\usepackage{paralist}
\usepackage{graphicx}
\usepackage{amsrefs}
\usepackage{amssymb}
\usepackage[matrix,arrow,ps]{xy}
\usepackage{epstopdf}
\usepackage{hyperref}
\newcommand{\real}{{\rm Re}\,}
\newcommand{\imag}{{\rm Im}\,}
\newcommand{\C}{\mathbb C}

\newcommand{\R}{\mathbb R}

\newcommand{\hC}{\hat{\mathcal C}}

\newcommand{\CN}{\mathbb{C}^N}
\newcommand{\CNp}{\mathbb{C}^{N^\prime}}

\newtheorem{theorem}{Theorem}[section]
\newtheorem{lemma}[theorem]{Lemma}
\newtheorem{corollary}[theorem]{Corollary}

\theoremstyle{remark}

\theoremstyle{definition}

\theoremstyle{definition}

\begin{document}
\author{Florian Bertrand}
\address{ American University of Beirut, Department of Mathematics,  P.O. box 11-0236, 
Riad El Sohl, Beirut 
1107 2020, Lebanon }
\email{fb31@aub.edu.lb}
\author{Giuseppe Della Sala}
\address{ American University of Beirut, Department of Mathematics,  P.O. box 11-0236, 
Riad El Sohl, Beirut 
1107 2020, Lebanon }\email{gd16@aub.edu.lb}
\author{Bernhard Lamel}
\address{Texas A\&M University at Qatar, Science program, PO Box 23874, Education City, Doha, Qatar }
\title[Extremal discs and Segre varieties]{Extremal discs and Segre varieties\\for real-analytic hypersurfaces in $\mathbb{C}^2$}
\email{bernhard.lamel@qatar.tamu.edu}

\subjclass[2010]{32V15, 32V40, 32T15, 32F45}

\keywords{}
\thanks{Research of the first two authors was  supported by a Research Group Linkage Programme from the Humboldt Foundation and a URB grant from the American University of Beirut.}
\thanks{Research of the third author was supported by the Austrian Science Fund FWF project I3472.}

\begin{abstract}
We show that if the Segre varieties of a 
strictly pseudoconvex hypersurface in $\mathbb{C}^2$ are 
extremal discs for the Kobayashi metric, then that 
hypersurface has to be locally spherical. In particular,
this gives yet another characterization of the unit sphere 
in terms of two important invariant families of objects
coinciding.
\end{abstract}

\maketitle 

\section{Introduction} 
\label{sec:introduction}
There is a deep link between complex analysis  in
a  smoothly bounded 
domain $\Omega \subset \C^N$ and 
the biholomorphically invariant (or short, CR) geometry 
of its boundary $b \Omega$. This link is 
known for 
strictly pseudoconvex smoothly bounded domains, where 
the relationship between one of the most important 
analytical objects associated to the domain, 
its Bergman Kernel function $K(z, \tilde z)$, and 
the CR invariants of its boundary (computable through the 
Chern-Moser normal form) have been well investigated: The 
asymptotic expansion of the Bergman kernel can be 
recovered from boundary invariants (and vice versa), a line 
of research instigated by Fefferman's work on the biholomorphically
invariant geometry of strictly pseudoconvex boundaries,
for which we refer the reader to \cite{MR684898} and also 
to Hirachi's work \cite{MR1745015}. 

In this paper, we will mostly deal with strictly
pseudoconvex domains $\Omega \subset\mathbb{C}^2$, 
whose boundary $b \Omega =:M$ we also assume to be 
real-analytic. In that case, 
there are two important families of boundary invariants: First, 
the Chern-Moser normal form which gives rise to {\em chains}, i.e.
families of biholomorphically invariant curves in $M$ 
intrinsically defined (we 
recall the basics of this in \autoref{sec:the_chern_moser_normal_form}); and the {Segre families} of invariant
complex curves defined near $M$. It is a theorem 
of Faran \cite{MR607109} that if the intersections
of the Segre family with $M$ and the chains
agree, then $M$ is locally biholomorphically equivalent to the 
unit sphere (Faran's result is valid in higher dimensions as well, 
but we will concentrate on $\mathbb{C}^2$ in this paper). 

There is yet another important family of invariant curves
 in a strictly pseudoconvex domain $\mathbb{C}^2$ which are associated 
to the Kobayashi pseudometric,
\[ k_{\Omega} (z,v) := \inf \left\{ a > 0 \colon f\in \mathcal{H} (\Delta, \Omega), 
f(0) = z, a f'(0)=v \right\}, \]
where $\Delta$ denotes the unit disc in $\C$.   
The corresponding integrated pseudodistance, 
the {\em Kobayashi pseudodistance}, gives us, for 
settings in which it actually is a distance,  a {\em biholomorphically invariant} distance notion, and with it, 
a natural hyperbolic geometry. A holomorphic disc $f \colon \Delta \to \Omega$
with $f(0) = z$
is said to be {\em extremal} for $(z,v) \in \Omega \times T_z \Omega$ 
if $f'(0) k_\Omega (z,v) = v$; extremal discs
are natural biholomorphic invariants of a bounded  domain just like the 
chains and Segre families discussed above. In 
the setting of the Kobayashi metric, things are a bit subtle: 
Even though  extremal discs   are proper and geodesics 
for the Kobayashi distance in strictly 
convex domains by the work of Lempert \cite{MR660145}, 
the hyperbolic geometry of strictly pseudoconvex domains is 
more complicated; in particular, a general extremal disc 
may fail to be proper.
However, work of Huang \cites{MR1310634,MR1260844}
shows that for $z$ sufficiently close to 
 $p \in b \Omega$ and for $v$ sufficiently close to the 
 complex tangent space $T^c_p b \Omega$, extremal 
 discs are again proper and complex geodesics. 
 We will
 consider such extremals as yet another biholomorphically invariant 
 family.


The relationship of the geometry of the boundary with 
the biholomorphically invariant hyperbolic geometry has been 
studied less, and in particular, the question answered by 
Faran about the Segre family and chains is open when 
asked about the Segre family and extremal discs. Our 
purpose in this paper is to settle this question (in $\mathbb{C}^2$).
In order to state our theorem, let us write for 
a neighbourhood $U$ of a point $p\in M$, where $M$ is a strictly 
pseudoconvex real
hypersurface, the decomposition $U= U_+ \cup (U\cap M) \cup U_-$
where $U_+$ lies on the pseudoconvex side of $M$.

%

\begin{theorem}\label{thm:mainlocal}
Let $M \subset \mathbb{C}^2$ be  a connected real-analytic hypersurface, 
$p\in M$. Assume
that there exist open neighbourhoods $U,V\subset \mathbb{C}^2$ of $p$ such that the Segre varieties
$S_q \subset U$, $q\in V_-$ are defined and such that 
${S_q \cap U_+}$ is an extremal disc for $U_+$ for every $q \in 
V_-$. 
Then $M$ is umbilical at every  strictly 
pseudoconvex point  of $ V\cap M$, and hence generically 
locally spherical.  
\end{theorem}


In particular, we also have the following characterization of
the unit ball: 

\begin{corollary}\label{thm:mainglob}
Let $\Omega \subset \mathbb{C}^2$ be  a bounded, simply connected strictly pseudoconvex domain with connected
real-analytic boundary, 
and assume that $U$ is a neighbourhood of $b \Omega$ such that 
$S_q \cap \Omega$ is defined for all $q \in U \cap \bar{\Omega}^c$. If 
 $S_q \cap \Omega$ is an extremal disc (for $\Omega$) for every $q\in U \cap \bar{\Omega}^c$, then 
$\Omega$ is biholomorphic to the the unit ball $\mathbb{B}$. 
\end{corollary}

We point out that the strict pseudoconvexity of the domain  is crucial as can be seen by considering 
the domain $\Omega=\{|z|^2+|w|^4<1\} \subset \mathbb{C}^2$, where it can been shown using \cite{MR1267598} that any  Segre variety $S_q \cap \Omega$ for $q \in \overline{\Omega}^c$ near $b\Omega$ is an extremal disc. 
Finally, we note that the local version \autoref{thm:mainlocal} is 
really stronger than
Corollary~\ref{thm:mainglob}, which 
 follows from \autoref{thm:mainlocal} after applying \cite[Theorem C]{MR1418903}. Indeed, one does not expect 
 a locally spherical hypersurface to be globally CR equivalent to the sphere, by examples 
 due to Burns and Shnider \cite{MR419857}; for these, by the localization theorem of 
 Huang \cite{MR1260844}, (small) stationary discs are exactly the intersections of (small) Segre varieties
 with the hypersurface.

{\bf Acknowledgment.} The authors would like to thank Ilya Kossovskiy for 
his interest in this work and the inspiring discussions arising from it. We also thank an anonymous referee and the editor for remarks and comments which helped
us to improve the paper. 
\newpage
\section{Preliminaries}
In this section, we recall some preliminaries which 
are needed later in the proof.

\subsection{The Chern-Moser normal form} 
\label{sec:the_chern_moser_normal_form}
Here we introduce the basics of 
the Chern-Moser normal form for real-analytic (or formal)
hypersurfaces in $\C^2$ which we need 
for our main argument; the normal form 
was introduced and its relation with the equivalence 
problem studied in \cite{MR0425155}. Recall that a 
germ of a real-analytic hypersurface $(M,p) \subset (\C^2, p)$ 
is defined by the vanishing locus of a germ of a 
real-valued real analytic function $\varrho (\tilde z, \bar{\tilde z}, 
\tilde w , \bar{\tilde w}) \in \C \left\{ {\tilde z, \bar{\tilde z}, 
\tilde w , \bar{\tilde w}} \right\}$. 
Strict pseudoconvexity of $M$ at $p$ means that 
the bordered complex Hessian of $\varrho$ satisfies
\[ \begin{vmatrix}
	0  & \varrho_{\tilde{z}} (p) & \varrho_{\tilde{w}} (p) \\
	\varrho_{\bar {\tilde{z}}} (p) & \varrho_{{\tilde{z}} \bar {\tilde{z}}} (p) & \varrho_{{\tilde{w}} \bar {\tilde{z}}} (p) \\
	\varrho_{\bar {\tilde{w}}} (p) & \varrho_{{\tilde{z}} \bar {\tilde{w}}} (p) & \varrho_{{\tilde{w}} \bar {\tilde{w}}} (p)
\end{vmatrix} \neq 0. \]

The {\em model hypersurface} for 
strictly pseudoconvex hypersurfaces in $\C^2$ 
is the Heisenberg hypersurface 
$\mathbb{H}  \colon \real w = |z|^2$. 
The automorphisms of the Heisenberg hypersurface
are given by linear fractional maps 
of the form
\[ (z,w) \mapsto \left( \lambda \frac{z + a w}{1 + 2  \bar a z + (|a|^2 + it) w}, |\lambda|^2 \frac{ w}{1 + 2  \bar a z + (|a|^2 + it) w} \right),
 \]
 for $(\lambda,a,t)\in \C^* \times \C \times \R$.

If $H(z,w)= (f(z,w), g(z,w))$ 
is a linear fractional map 
of this form, $H$ can be determined from 
$f_z (0)$, $f_w(0)$, and $\imag g_{ww} (0)$. Thus, 
the jet map $j_0^2 \colon {\rm Aut} (\mathbb{H},0) \to 
G_0^2 (\C^2)$,  $H\mapsto j_0^2 H$, is an 
injective homeomorphism onto its image 
$\Gamma$. The group ${\rm Aut} (\mathbb{H},0)$ is 
of maximal dimension amongst all
automorphism groups of strictly pseudoconvex
hypersurfaces in $\C^2$ and therefore gives  
the natural space for
parameters of a normal form for the family 
of strictly pseudoconvex hypersurfaces under the action 
of the group of local 
biholomorphisms.

The  
{\em Chern-Moser normal form} gives, for any choice 
of a parameter $\gamma\in \Gamma$,  
a change of coordinates $(\tilde z, \tilde w )= H_M^p (z,w,\gamma) 
= (f(z,w,\gamma), g(z,w,\gamma)) $, which is 
uniquely determined under the following conditions: 
\begin{compactitem}
\item in the
new coordinates $(z,w)$, $p$ is the origin (i.e. $H_M^p (0, \gamma) = p$);  
\item  the defining equation 
of $M$ in the new coordinates has the form
\[ \real w = \varphi(z ,\bar z, \imag w) = \sum_{j,k} \varphi_{j,k} (\imag w) z^j \bar z^k,  \]
where $\varphi$ satisfies the {\em normalization 
conditions} 
\[ \begin{aligned}
\varphi_{j,0} (t) &= 0,  & j \geq 0, \\
\varphi_{1,1} (t) &= 1,  \\ 
\varphi_{j,1} (t) &= 0, & j\geq 2, \\ 
\varphi_{2,2} (t) = \varphi_{2,3} (t) = \varphi_{3,3} (t) &= 0.
\end{aligned}   \]
\end{compactitem}
One can in addition require that 
$j_0^2 H_M^p (\cdot, \gamma) = \gamma \cdot j_0^2 H_M^p(\cdot, {\rm id})$.
The {\em chains} through $p$ are the (parametrized) curves given by 
$t\mapsto H_M^p (0,t,\gamma)$. 
 
The term $\varphi_{j,k} (t) z^j \bar z^k$  is said to be of {\em type $(j,k)$}.
The lowest order (nontrivial) invariant terms in 
the normal form are therefore the terms of 
type $(4,2)$ (and $(2,4)$), $\varphi_{4,2}$ and $\varphi_{2,4}$ respectively; 
they correspond to {\em Cartan's cubic tensor} from \cites{MR1556687,MR1553196}.  
The condition that  $\varphi_{2,4} (0) = \varphi_{4,2} (0) = 0$
is invariant under different choices of $\gamma$, and if 
it is  satisfied for one (and hence, all) $\gamma \in \Gamma$ 
we say that the point $p$ is {\em umbilic}. 

Umbilicity is quite different in higher dimensions, which is why we 
concentrate on the two-dimensional case here. Umbilicity at a 
point means that 
the order of approximation of a given strictly pseudoconvex  with 
the model hypersurface is higher than generically expected, and the model
hypersurface is the only one which is everywhere umbilical: If $M$ is 
a strictly pseudoconvex real-analytic
 hypersurface, and if $p\in M$ has the property 
that it possesses a neighbourhood consisting of umbilical points, then 
there exists a neighbourhood $U\subset M$ of $p$ which is biholomorphically equivalent to 
a piece of the model hypersurface. The same holds for smooth $M$ if one replaces
``biholomorphically equivalent'' by ``$\mathcal{C}^\infty$-CR equivalent'', and 
the fact can simply be stated by saying that every umbilical hypersurface 
is locally spherical. 

\subsection{Segre varieties} 
\label{sec:segre_varieties}
Let $M\subset \CN$ be a real-analytic hypersurface, defined locally
at $p\in M$ by a real-analytic equation $\varrho(Z,\bar Z) = 0$. To be more
precise, here we assume that 
$\varrho (Z, W) $ is holomorphic on $U \times  U^* \subset \C^{2N}$, where
$U^* = \left\{ Z \colon \bar Z \in U \right\}$, and $\varrho_W (Z,W) \neq 0$ for 
$(Z,W) \in U\times U^*$. then 
one can define the Segre variety associated to the point $q$ (in 
a suitable neighbourhood $V$ of $p$) 
by 
\[ S_q = \left\{ z \in U \colon \varrho(z, \bar q) =0 \right\}, \quad q\in V. \]
For good choices of $U$ and $V$, for every $q\in V$, the variety 
$S_q \subset U$ is a connected, smooth complex hypersurface in $U$. One can 
check that for $p\in M$, we have that $T_p S_p = T_p^c M$. Actually, 
a bit more is true: given $p\in M$, and a real-analytic 
curve $\Gamma \subset M$ through $p$ transverse to $T_p^c M$, 
one can choose coordinates 
$Z = (z_1, \dots ,z_{N-1}, w)$ near $p$ such that in these coordinates,
$p = 0$, and  
$\Gamma = \left\{ (0,\dots, 0, it) \right\}$ and for small $s$, we have that
$S_{(0,s)} = \left\{ w = -s \right\}$ (see e.g. \cite[Lemma 4.1]{Lamel:2007th}).

The importance of Segre varieties is that they transform very 
nicely with respect to holomorphic maps: If $H$ is a germ 
of a holomorphic map taking a real-analytic submanifold $M\subset \CN$
into a real-analytic submanifold $M' \subset \CNp$, then $H(S_p) \subset S'_{H(p)}$, 
where $S'_{q'}$ denotes the Segre variety of $M'$ (associated to $q'$). 

\subsection{Stationary discs} 
\label{sub:extremal_and_stationary_discs}

Let $M=\{\varrho=0\}$ be a smooth  hypersurface in  $\C^N$.     
A disc $f: \overline{\Delta} \to \C^N$ continuous up to $b\Delta$ and holomorphic in $\Delta$  is  {\it attached} to $M$  if $f(b\Delta)\subset M$. 
Following Lempert \cite{MR660145}, such a map is called {\it stationary} if there exists a continuous function 
$c:b\Delta\to\R^*$ such that the map $\zeta c(\zeta)\partial \varrho(f(\zeta))$, defined on $b \Delta$, 
extends holomorphically into $\Delta$. Here $\partial\varrho=\partial_Z\varrho$ denotes the complex gradient of $\varrho$.  Equivalently, one 
can require that $f$ allows for a meromorphic lift with a pole
of order at most $1$ to the conormal bundle $\mathcal{N}^* M$. For details
on that, we refer the reader e.g. to \cite{MR3158006}. We will only deal with small discs $f$.


\subsection{Spaces of functions with parameters} 
Here, we  recall  spaces of functions with parameters defined in \cite{MR0149505}, 
suitable for the study of 
deformations of Riemann maps  \cite[Corollary 9.4]{MR3118395} we are going to use. 
Let $I=[0,1] \subset \R$, let $\Omega$ be a bounded open set in a Euclidean space and let  $k,j\geq0$ be integers and let $0\leq\alpha<1$,
We denote by $\mathcal{C}^{k+\alpha}(\overline \Omega)$ the standard H\"older space with its usual norm 
$|\cdot|_{k+\alpha}$.
Define $\hC^{k+\alpha,j}(\overline \Omega,I)$
to be the set of functions $f$ defined on $\overline \Omega\times I$  such that
 for all integers $0\leq l\leq j$,
the map $t\mapsto \partial_t^l f(.,t)$ is continuous from $I$ into
$\mathcal{C}^{k}(\overline \Omega)$ and such that
$$
\|f\|_{k+\alpha,j}:=\max_{0\leq l\leq j} \sup_{t\in I} |\partial_t^lf(\cdot,t)|_{k+\alpha}<\infty.
$$
We now define
$$\mathcal{C}^{k+\alpha, j}(\overline \Omega, I):=\bigcap_{0\leq l\leq j}\hC^{k-l+\alpha,l}(\overline \Omega, I),$$
and 
$$|f|_{k+\alpha,j }:=\max_{0\leq l\leq j} \|f \|_{k-l+\alpha,l}.$$
As pointed out in \cite{MR3118395},  we have the following inclusion:
 \begin{equation}\label{eqincl}
 \mathcal{C}^{k+\alpha}(\overline \Omega\times  I) \subset \mathcal{C}^{k+\alpha,k}(\overline \Omega,I).
 \end{equation}
In the present paper we will also need:
\begin{lemma}\label{lemincl}
The following inclusion holds
$$  \mathcal{C}^{k+1+\alpha,k+1}(\overline \Omega,I)  \subset \mathcal{C}^{k}(\overline \Omega\times  I).$$
\end{lemma}
\begin{proof}
Let $f \in \mathcal{C}^{k+1+\alpha,k+1}(\overline \Omega,I)$. We first note that $f$ is $k$-times differentiable. Now let $p,q\geq 0$ 
be two integers with $p+q=k$ and let  
 $(x,t), (x',t') \in \overline \Omega \times I$ sufficiently
  close to each other. We have 
\[\begin{aligned}
|\partial_x^p\partial_t^qf(x,t)&-\partial_x^p\partial_t^qf(x',t')| \\ &\leq |\partial_x^p\partial_t^qf(x,t)-\partial_x^p\partial_t^qf(x,t')|+|\partial_x^p\partial_t^qf(x,t')-\partial_x^p\partial_t^qf(x',t')| \\
& \leq   \sup_{s\in I}|\partial_x^p\partial_t^{q+1}f(x,s)| |t-t'|+ \sup_{y\in [x,x']}|\partial_x^{p+1}\partial_t^{q}f(y,t')| \|x-x'\|\\
& \leq   \sup_{s\in I}|\partial_t^{q+1}f(\cdot,s)|_{k+1+\alpha} |t-t'|+ \sup_{s \in I}|\partial_t^{q}f(\cdot,s)|_{k+1+\alpha} \|x-x'\|\\
& \leq   |f|_{k+1+\alpha,k+1} |t-t'|+ |f|_{k+1+\alpha,k+1} \|x-x'\|,
\end{aligned}    \]
which proves the lemma.
\end{proof}







\section{A further result and proof of the main theorem} 
\label{sec:proof_of_the_main_result}

In order to prove 
\autoref{thm:mainlocal}, we shall make use of the following result 
which is valid for finitely smooth real hypersurfaces 
in $\C^2$.   We use the following  convention: We write $O(|z|^k)$ (resp. $O(t^k)$) to denote a function of class at least  $\mathcal{C}^{k}$ which is bounded by $|z|^k$ (resp. $t^k$) up to a multiplicative constant. 
\begin{theorem} \label{thm:maintechnical} Let $S\subset \C^2$ be a $\mathcal{C}^{8+\alpha}$-smooth
real hypersurface through the origin, with $\alpha>0$, whose defining equation (near the origin) can
be written in the form 
\[\varrho(z,w,\overline z,\overline w)= \real w - |z|^2 + A z^2\overline z^4 + \overline A z^4\overline z^2 +  \imag w\,h(z,\overline z, \imag w) + g(z,\overline z), \]
with $g(z, \overline{z}) = O(|z|^7)$. If the 
discs $ \Omega_t =\{w=t^2\}\cap \{\varrho>0\}  $, for small 
$t \in \R$, are stationary, then $A = 0$.
\end{theorem}
\autoref{thm:mainlocal} is 
 a straightforward consequence of 
\autoref{thm:maintechnical}: the result of 
Huang \cite{MR1260844} already mentioned in the introduction shows 
that the extremal discs we consider are actually stationary, 
 and the results of 
Chern and Moser summarized in \autoref{sec:the_chern_moser_normal_form} show 
that there exists a spherical neighbourhood of $0$ in $S$. The rest of this 
section is devoted to the proof of \autoref{thm:maintechnical}, which is going to be developed 
in a series of lemmas. 


First note that if 
for any $t\in \mathbb R$ we write $S_t=\{w=t^2\}\cap S$, then for 
 $|t|>0$ small enough, $S_t$ is a closed curve (of 
class $\mathcal{C}^{8+\alpha}$) contained in $S$, bounding $\Omega_t$. For $t$ small enough, 
we will have that $(0,t^2) \in \Omega_t$.
We denote by 
$\pi^1$ the projection onto 
the first coordinate and for all $|t|>0$ small enough,
we consider 
$R_t:\Delta\to \pi^1(\Omega_t)$ the (uniquely determined) Riemann map such that $R_t(0)=0$, $R'_t(0)>0$. Define $f_t:\Delta\to\mathbb C^2$ as $f_t(\zeta)=(R_t(\zeta),t^2)$. By construction, each $f_t$ is an analytic disc attached to $S$. In the following, 
for the sake of notational simplicity, we will 
identify $\Omega_t$ with $\pi^1 (\Omega_t)$ (as 
well as $S_t$ with $\pi^1 (S_t)$).

By definition, $f_t$ is stationary if and only if there exists a continuous function $a_t:b\Delta\to \mathbb R^+$ and holomorphic functions $\widetilde z_t,\widetilde w_t\in \mathcal O(\Delta)\cap C(\overline \Delta)$ satisfying
\begin{equation} \label{eq:stat}
\begin{aligned} 
\widetilde z_t(\zeta) &= \zeta a_t(\zeta) \frac{\partial \varrho}{\partial z}(R_t(\zeta), t^2, \overline{R_t(\zeta)}, t^2 ) \\ 
\widetilde w_t(\zeta)&= \zeta a_t(\zeta) \frac{\partial \varrho}{\partial w}(R_t(\zeta), t^2, \overline{R_t(\zeta)}, t^2 ) 
\end{aligned} 
\end{equation}
for all $\zeta\in b\Delta$.

Let now $R_t^{-1}:\Omega_t\to \Delta$ be the inverse of the Riemann map. Note that $R_t^{-1}$ is smooth of class $\mathcal{C}^{8+\alpha}$  
up to the boundary  $S_t $ by Kellogg's theorem \cite{MR1500802} (see e.g. the book of Pommerenke \cite{MR1217706}). We 
also note that we can write $R_t^{-1}(z)=ze^{\varphi_t(z)}$ for a suitable holomorphic function $\varphi_t:\Omega_t\to \Delta$, where $\varphi_t$ is
again smooth of class $\mathcal{C}^{8+\alpha}$ up to $S_t$. 
Applying (\ref{eq:stat}) for $\zeta=R_t^{-1}(z)$ we obtain
\[
\begin{aligned} 
\widetilde z_t(R_t^{-1}(z)) &= z e^{\varphi_t(z)} a_t(R_t^{-1}(z)) \frac{\partial \varrho}{\partial z}(z, t^2, \overline{z}, t^2 ) \\ 
\widetilde w_t(R_t^{-1}(z)) &= z e^{\varphi_t(z)} a_t(R_t^{-1}(z)) \frac{\partial \varrho}{\partial w}(z, t^2, \overline{z}, t^2 ) 
\end{aligned} 
\]
for all $z\in b\Omega_t=S_t$. Putting $b_t(z)=a_t(R_t^{-1}(z))$, $Z_t(z)=e^{-\varphi_t(z)}\widetilde z_t(R_t^{-1}(z))$ and $W_t(z)=e^{-\varphi_t(z)}\widetilde w_t(R_t^{-1}(z)) $ we can rewrite the system as
\begin{equation} \label{eq:stat2}
\begin{aligned} 
Z_t(z) &= z  b_t(z) \frac{\partial \varrho}{\partial z}(z, t^2, \overline{z}, t^2 ) \\ 
W_t(z) &= z  b_t(z) \frac{\partial \varrho}{\partial w}(z, t^2, \overline{z}, t^2 ) 
\end{aligned} 
\end{equation}
for $z\in S_t$. Here $b_t$ is a continuous positive function on $S_t$ and the functions $Z_t, W_t$ extend holomorphically to $\Omega_t$.

We will use systematically the following fact: a continuous function $f:S_t\to\mathbb C$ extends holomorphically to $\Omega_t$ if and only if it satisfies the \emph{moment conditions}
\[\int_{S_t} z^m f(z) dz =0 \ \ \mbox{ for all } m\geq 0.\]
Though this fact is well-known, we provide a proof for the convenience of the reader. Denote by $Cf$ the Cauchy transform
\[Cf(z)=\frac{1}{2\pi i}\int_{S_t} \frac{f(\zeta)}{\zeta-z}d\zeta.\] 

By Plemelj's formula, $f$ extends holomorphically if and only if $Cf(z)=0$ for $z\not\in \overline\Omega_t$. For 
any fixed $z$ outside $\overline\Omega_t$, $1/(\zeta-z)$ can be approximated uniformly by polynomials on $S_t$ by Runge's theorem. Since the moment conditions mean that the integral of $f$ against any holomorphic polynomial vanishes, we deduce that $f$ extends holomorphically whenever it satisfies the moment conditions. The opposite implication is a consequence of Cauchy's integral formula.

Consider the scaling $\Lambda_t: \mathbb{C}\to  \mathbb{C}$ defined by $\Lambda_t(z)=z/t$. We set $\widetilde \Omega_t=\Lambda_t(\Omega_t)$, $\widetilde S_t=\Lambda_t(S_t)$, and $\widetilde \Omega_0 = \Delta$ (with 
$\widetilde S_0 = b \Delta$).  A change of variables in the above integral implies that 
$f:S_t\to\mathbb C$ extends holomorphically to $\Omega_t$ if and only if it satisfies the moment conditions
\[\int_{\widetilde S_t} z^m f(tz) dz =0 \ \ \mbox{ for all } m\geq 0.\]
In order to compute integrals of this kind, we use polar coordinates $(r,\theta)$ and parametrize the curve $\widetilde S_t$
according to the following Lemma.




\begin{lemma}\label{lem:polar} If we parametrize the curve $\widetilde S_t$ as $\theta\mapsto r(\theta,t) e^{i\theta}$, then the function 
 $r$  is of class $\mathcal{C}^{6+\alpha}$ 
 in both variables in  a full neighbourhood of $(0,0) \in \mathbb{R}^2$, and 
 can be written as  
 $$r(\theta,t)=1+k(\theta) t^4 + r_5(\theta,t)$$ where $k(\theta)=  \real\, (A e^{-2i\theta})$ and $r_5 (\theta, t) = O(|t|^5)$.
\end{lemma}
\begin{proof}
The function $r$ satisfies $\varrho(tr(\theta,t)e^{i\theta}, t^2, tr(\theta,t)e^{-i\theta}, t^2)\equiv 0$, i.e.\
\begin{equation}\label{eq:polar}
t^2 - t^2 r^2(\theta,t) + 2k(\theta) t^6 r^6(\theta,t)  + g(tr(\theta,t)e^{i\theta},tr(\theta,t)e^{-i\theta})= 0. 
\end{equation}
Since $g(z,\overline z)=O(|z|^7)$ we have $g(tr(\theta,t)e^{i\theta},tr(\theta,t)e^{-i\theta})=t^2 G(\theta,t)$ where  $G$ is of class $\mathcal{C}^{6+\alpha}$ and satisfies
$G(\theta,t) = O(t^5)$. This allows us to rewrite (\ref{eq:polar}) as
\begin{equation*}
1 - r^2(\theta,t) + 2k(\theta) t^4 r^6(\theta,t)  +  G(\theta,t)= 0. 
\end{equation*}
The implicit function theorem allows us to solve this 
equation with a unique $r$ of class $\mathcal{C}^{6+\alpha}$ satisfying $r(\theta,0)=1$. Taking successive derivatives it is immediate that $\frac{\partial^j r}{\partial t^j}(\theta,0)=0$ for $j=1,2,3$ and
\[-2r(\theta,t)\frac{\partial^4 r}{\partial t^4}(\theta,t) + 4!\cdot 2k(\theta)  +O(t)= 0,\]
so that $\frac{\partial^4 r}{\partial t^4}(\theta,0)=4! k(\theta)$. This concludes the proof of the lemma.
\end{proof}

The lemma above allows us to extend the boundary parametrization of $\widetilde S_t$ to the interior of the unit disc, to obtain a family of diffeomorphisms $\Gamma_t=\Gamma(\cdot,t):\overline{\Delta} \to \overline{\widetilde \Omega_t}$ which is $\mathcal{C}^{6+\alpha}$ in both variables  $z$ and $t$ and equal to the identity for $t=0$. We shall (if necessary) rescale 
with a map of the form $(z,w) \mapsto (\lambda z, \lambda^2 w)$ to have that $\Gamma \in \mathcal{C}^{6+\alpha}(\overline{\Delta} \times I)$, where $I=[-1,1]$. 
\begin{lemma}\label{lemana}
The map $R_t^{-1}\circ \Lambda_t^{-1} \circ \Gamma_t: \overline{\Delta} \to \overline{\Delta}$ is in $\mathcal{C}^{5}(\overline{\Delta} \times I)$ and is, for $t=0$, the identity.   
\end{lemma}
\begin{proof}
Since the family of diffeomorphisms 
$\Gamma_t:\overline{\Delta} \to \overline{\widetilde \Omega_t}$  is $\mathcal{C}^{6+\alpha}$ in both variables $z$ and $t$, we 
can apply \eqref{eqincl} to obtain that 
 $\Gamma \in \mathcal{C}^{6+\alpha,6} (\overline{\Delta} \times I)$.
 By
 \cite[Corollary 9.4]{MR3118395}, we have $R_t^{-1}\circ \Lambda_t^{-1} \circ \Gamma_t \in \mathcal{C}^{6+\alpha,6}(\overline{\Delta}\times I)$ and 
 by Lemma \ref{lemincl} it follows that $R_t^{-1}\circ \Lambda_t^{-1} \circ \Gamma_t \in \mathcal{C}^{5} (\overline{\Delta} \times I)$.
 
 Since $\widetilde \Omega_0=\Delta$ and since the Riemann map $R_t^{-1}\circ \Lambda_t^{-1}: \overline{\widetilde \Omega_t} \to \overline{\Delta}$ is chosen in such a way that  $R_t(0)=0$, $R'_t(0)>0$ and so $R_t^{-1}\circ \Lambda_t^{-1}(0)=0$ and $\left(R_t^{-1}\circ \Lambda_t^{-1}\right)'(0)>0$, it follows that $R_t^{-1}\circ \Lambda_t^{-1}$ is the identity for $t=0$. Finally note that by definition, $\Gamma_0$ is the identity.        
\end{proof}

We will now apply the moment conditions to the second equation in (\ref{eq:stat2}):
\[\int_{\widetilde S_t} z^m \left(z b_t(tz) \frac{\partial \varrho}{\partial w}(tz,t^2,t\overline z, t^2)\right) dz =0 \ \ \mbox{ for all } m\geq 0\]
or equivalently
\[\int_{\widetilde S_t} z^j b_t(tz) \frac{\partial \varrho}{\partial w}(tz,t^2,t\overline z, t^2) dz =0 \ \ \mbox{ for all } j\geq 1.\]
Computing $\partial \varrho/\partial w$ we get
\[\frac{\partial \varrho}{\partial w}(z,w,\overline z, \overline w) = \frac{1}{2}-\frac{i}{2}h(z,\overline z,\imag\, w) + \imag\, w \frac{\partial}{\partial w} (h(z,\overline z, \imag\, w))\]
so that
\[\frac{\partial \varrho}{\partial w}(tz,t^2,t\overline z, t^2) = \frac{1}{2}-\frac{i}{2}h(tz,t\overline z,0). \]
Hence $b_t(tz)$ must satisfy
\[\int_{\widetilde S_t} z^j b_t(tz) \left(\frac{1}{2}-\frac{i}{2}h(tz,t\overline z,0)\right) dz =0 \ \ \mbox{ for all } j\geq 1.\]
Using the parametrization $\theta\mapsto r(\theta,t) e^{i\theta}$  for $\widetilde S_t$ the integral becomes
\begin{equation}
	\label{e:beforeplugging}
	\int_0^{2\pi} r^j e^{ij\theta} b_t(tre^{i\theta}) \left(\frac{1}{2}-\frac{i}{2}h(tr e^{i\theta},tr e^{-i\theta},0)\right) \left(\frac{\partial r}{\partial \theta} + ir\right)e^{i\theta}d\theta =0 \ \ \mbox{ for all } j\geq 1,
\end{equation}
where we write $r=r(\theta,t)$ for brevity. 

For a continuous function $a_t:b\Delta\to \mathbb R^+$ satisfying (\ref{eq:stat}), we define 
\begin{equation*}
c(\theta,t)=b_t(tr(\theta,t)e^{i\theta})=b_t(\Lambda_t^{-1} \circ \Gamma_t (e^{i\theta}))=a_t(R_t^{-1}\circ \Lambda_t^{-1} \circ \Gamma_t(e^{i\theta})).
\end{equation*}
\begin{lemma}
There is a choice of $a_t$ such that the function $c(\theta,t)$ is $\mathcal{C}^{4}$ in a neighbourhood of $[0,2\pi]\times \{0\}$ and satisfies
$\int_0^{2\pi} c(\theta,t) dt =1$ for all $t\neq 0$ small enough.
\end{lemma}

We note that the normalization condition can of course be assumed because
the sign of $c$ is fixed. The point of the Lemma is the smoothness of 
the function $a_t$.

\begin{proof}

Recall that by Lemma \ref{lemana}, $R_t^{-1}\circ \Lambda_t^{-1} \circ \Gamma_t$ is of class $\mathcal{C}^{5}$. 
From Pang \cite{MR1250257}, 
if $f_t(\zeta)=(R_t(\zeta),t^2)$  is stationary and satisfies  (\ref{eq:stat}) for a continuous function $a_t:b\Delta\to \mathbb R^+$, then  $a_t$ is a positive multiple of $\widehat{a_t}$, which is defined  for $\zeta \in b\Delta$  by 
\begin{equation*}
\frac{1}{\widehat{a_t}(\zeta)}=\zeta\partial \varrho (f_t(\zeta))\cdot f_t'(\zeta).
\end{equation*}
 First the map 
$$\partial \varrho (f_t(R_t^{-1}\circ \Lambda_t^{-1} \circ \Gamma_t(e^{i\theta})))=
\partial \varrho ( t \Gamma_t(e^{i\theta}), t^2)$$ 
is $\mathcal{C}^{6+\alpha}$.
Note that by the chain rule, we have 
\begin{eqnarray*}
\frac{d}{d\theta} \left(t\Gamma_t(e^{i\theta})\right) &=&\frac{d}{d\theta} R_t\circ \left(R_t^{-1}\circ \Lambda_t^{-1} \circ \Gamma_t(e^{i\theta})\right)\\
&=& R_t' \left(R_t^{-1}\circ \Lambda_t^{-1} \circ \Gamma_t(e^{i\theta})\right)\frac{d}{d\theta} \left(R_t^{-1}\circ \Lambda_t^{-1} \circ \Gamma_t(e^{i\theta})\right)
\end{eqnarray*}
since $R_t$ is holomorphic. It follows that
$$f_t' (R_t^{-1}\circ \Lambda_t^{-1} \circ \Gamma_t(e^{i\theta}))=\left(\frac{\frac{d}{d\theta}\left(t \Gamma_t(e^{i\theta})\right)}{\frac{d}{d\theta}\left(R_t^{-1}\circ \Lambda_t^{-1} \circ \Gamma_t(e^{i\theta}))\right)},0\right).$$
Since $\left(R_t^{-1}\circ \Lambda_t^{-1} \circ \Gamma_t(e^{i\theta})\right)$ is $\mathcal{C}^{5}$  and  equal to 
$e^{i\theta}+O(t)$ by Lemma \ref{lemana}, the function $\frac{d}{d\theta}\left(R_t^{-1}\circ \Lambda_t^{-1} \circ \Gamma_t(e^{i\theta})\right)$ is $\mathcal{C}^{4}$ and  equal to $ie^{i\theta}$ for $t=0$. This shows that the function $f_t' (R_t^{-1}\circ \Lambda_t^{-1} \circ \Gamma_t(e^{i\theta}))$ is $\mathcal{C}^{4}$ and therefore that 
$1/\widehat{a_t}(R_t^{-1}\circ \Lambda_t^{-1} \circ \Gamma_t(e^{i\theta}))$
  is $\mathcal{C}^{4}$.
Finally, note that since  $t\Gamma_t(e^{i\theta})=te^{i\theta}+O(t^2)$ and $\partial_z\varrho(z,w)=\overline{z}+O(|z|^5)$, we have 
  $$\frac{\frac{d}{d\theta}\left(t \Gamma_t(e^{i\theta})\right)}{\frac{d}{d\theta}\left(R_t^{-1}\circ \Lambda_t^{-1} \circ \Gamma_t(e^{i\theta}))\right)}=\frac{ite^{i\theta}+O(t^2)}{ie^{i\theta}+O(t)}=t+O(t^2),$$
  and 
 $$\partial_z\varrho(R_t^{-1}\circ \Lambda_t^{-1} \circ \Gamma_t(e^{i\theta}),t^2)=te^{-i\theta}+O(t^2)$$  
from which it follows directly that 
   $$\frac{1}{\widehat{a_t}(R_t^{-1}\circ \Lambda_t^{-1} \circ \Gamma_t(e^{i\theta}))}=  t^2+O(t^3).$$
 The function  $\tilde a_t=t^2\widehat{a_t}$ satisfies all of the 
 required properties and can be rescaled so that 
 $\int_0^{2\pi} c(\theta,t) dt =1$ for all $t\neq 0$ small enough
 without changing the smoothness of $c$.  
  \end{proof}

Since $h(z,\overline z,0)=O(|z|^6)$, using Lemma \ref{lem:polar} we deduce that $h(tre^{i\theta}, tre^{-i\theta},0)=O(t^6)$, and furthermore
\begin{align}\label{eq:powder}
r(\theta,t)^j  & = 1 + j k(\theta) t^{4} + O(t^{5}), \\
\frac{\partial r}{\partial \theta}(\theta,t) & = \frac{dk}{d\theta}(\theta)t^4 + O(t^5).
\end{align}

 Thus we can rewrite \eqref{e:beforeplugging} as
 \[\int_0^{2\pi} c e^{i(j+1)\theta}  (1 + j k t^{4} + O(t^{5})) \left(\frac{1}{2}+O(t^6)\right) \left(i +  \left(\frac{dk}{d\theta}+ik\right) t^4  + O(t^5)\right)d\theta =0\]
for all $j\geq 1$. Further developing the products we obtain
\begin{equation}\label{eq:devpro}
\int_0^{2\pi} e^{i(j+1)\theta} c(\theta,t) \left(i + \left(i(j+1) k(\theta) +\frac{dk}{d\theta}(\theta)\right) t^{4} + O(t^{5})\right) d\theta =0
\end{equation}
for all $j\geq 1$. For all $|t|$ small enough we expand the function $c(\cdot,t)$ in its Fourier series $c(\theta,t)=\sum_{k=-\infty}^{+\infty}\gamma_k(t)e^{ik\theta}$, where $\gamma_k$ is $\mathcal{C}^4$ for all $k\in \mathbb Z$, $\gamma_{-k}=\overline \gamma_k$ and $\gamma_0(t)\equiv 1$ due to our normalization. We insert this series in (\ref{eq:devpro}) and ignore for the moment the precise expression of the factor multiplying $t^4$:
\begin{equation*}
\int_0^{2\pi} e^{i(j+1)\theta} \sum_{k=-\infty}^{+\infty}\gamma_k(t)e^{ik\theta}  d\theta = O(t^4), 
\end{equation*}
which means that 
\begin{equation}\label{eq:order4}
\overline \gamma_{j+1}(t)=O(t^4) \ \ \mbox{ for all } j\geq 1. 
\end{equation}
Next, we write also $k(\theta)=A e^{-2i\theta}/2 + \overline A e^{2i\theta}/2$ and $\frac{dk}{d\theta}(\theta)=-iA e^{-2i\theta}+i\overline A e^{2i\theta}$ as Fourier polynomials, so that  
\begin{equation*}
i(j+1) k(\theta) +\frac{dk}{d\theta}(\theta)= i\frac{j-1}{2}A e^{-2i\theta} + i\frac{j+3}{2}\overline A e^{2i\theta}
\end{equation*}
and we take the fourth derivative of (\ref{eq:devpro}) with respect to $t$: 
\begin{equation*}\begin{aligned}
\sum_{\ell=0}^4\binom{4}{\ell}\int_0^{2\pi}& \frac{\partial^\ell c}{\partial t^\ell} e^{i(j+1)\theta} \cdot \\ & \cdot \left(\delta_{4}^\ell + \frac{4! t^{\ell}}{\ell!}\left(\frac{j-1}{2}A e^{-2i\theta} + \frac{j+3}{2}\overline A e^{2i\theta}\right) + O(t^{\ell+1})\right) d\theta =0 \end{aligned}
\end{equation*}
where $\delta_4^\ell=1$ if $\ell=4$ and $\delta_4^\ell=0$ otherwise, and we divided by the common factor $i$. 
Replacing $\frac{\partial^\ell c}{\partial t^\ell}(\theta,t)$ with its Fourier series, we see that 
\begin{equation*}\begin{aligned}
&\sum_{\ell=0}^4\binom{4}{\ell} \cdot \\ &\cdot \left(\frac{d^\ell \overline \gamma_{j+1}}{dt^\ell}(t)\delta_{4}^\ell + \frac{4!}{\ell!}\left(\frac{j-1}{2}A \frac{d^\ell \overline \gamma_{j-1}}{dt^\ell}(t) + \frac{j+3}{2}\overline A \frac{d^\ell \overline \gamma_{j+3}}{dt^\ell}(t)\right) t^{\ell} + O(t^{\ell+1})\right)  =0 \end{aligned}
\end{equation*}
for all $j\geq 1$. In particular for $j=1$
\begin{equation*}
\sum_{\ell=0}^4\binom{4}{\ell}\left(\frac{d^\ell \overline \gamma_{2}}{dt^\ell}(t)\delta_{4}^\ell + \frac{4!}{\ell!}\left( 2\overline A \frac{d^\ell \overline \gamma_{4}}{dt^\ell}(t)\right) t^{\ell} + O(t^{\ell+1})\right)  =0.
\end{equation*}
All the terms in the previous sum are $O(t)$ except for $\ell=0$ and $\ell=4$:
\begin{equation*}
\frac{d^4 \overline \gamma_{2}}{dt^4}(t) + 4!\cdot 2\overline A \overline \gamma_{4}(t) = O(t),
\end{equation*}
which by (\ref{eq:order4}) implies that 
\begin{equation}\label{eq:gamma2}
\frac{d^4 \overline \gamma_{2}}{dt^4}(t)=O(t).
\end{equation}

We now turn to the first equation in (\ref{eq:stat2}). The moment conditions read
\[\int_{\widetilde S_t} z^j b_t(tz) \frac{\partial \varrho}{\partial z}(tz,t^2,t\overline z, t^2) dz =0 \ \ \mbox{ for all } j\geq 1,\]
and since
\[\frac{\partial \varrho}{\partial z} (z,w,\overline z,\overline w)=-\overline z+2A z \overline z^4+4\overline A z^3 \overline z^2+\imag w\,\frac{\partial h}{\partial z}(z,\overline z, \imag w) + \frac{\partial g}{\partial z}(z,\overline z), \ \ \mbox{so that }\]
\[\frac{\partial \varrho}{\partial z} (tz,t^2,t\overline z,t^2)=-t\overline z+2A t^5 z \overline z^4+4\overline A t^5z^3 \overline z^2+ \frac{\partial g}{\partial z}(tz,t\overline z),\]
by using the parametrization $\theta\to re^{i\theta}$ the integral turns into
\[\int_0^{2\pi}  e^{i(j+1)\theta} b_t \left(-tr^{j+1}e^{-i\theta}+t^5r^{j+5}(2A  e^{-3i\theta}+4\overline A  e^{i\theta})+ r^j\frac{\partial g}{\partial z}\right)\left(\frac{\partial r}{\partial \theta} + ir\right)d\theta=0 \] 
for all $j\geq 1$. Since $\frac{\partial g}{\partial z}(z,\overline z)=O(|z|^6)$ we have $r^j\frac{\partial g}{\partial z}(tre^{i\theta},tr e^{-i\theta})=O(t^{6})$. We use now (\ref{eq:powder}) and recall that $b_t(tre^{i\theta})=c(\theta,t)$ to rewrite the previous equation as 
\[\int_0^{2\pi}  e^{i(j+1)\theta} c(\theta,t)\left(-(t+(j+1)k(\theta)t^5)e^{-i\theta} +t^5(2 A  e^{-3i\theta}+4\overline A  e^{i\theta}) +O(t^6)\right)\cdot\] \[\cdot \left(i+ \left(\frac{dk}{d\theta}(\theta)+ik(\theta)\right) t^4  + O(t^5)\right)d\theta =0 \ \ \mbox{ for all } j\geq 1.\]
We substitute the expression of $k(\theta)$ and divide by $-t$ to obtain
\[\int_0^{2\pi}  e^{i(j+1)\theta} c(\theta,t)\left(e^{-i\theta}+\left(\frac{j-3}{2} A e^{-3i\theta}+\frac{j-7}{2}\overline  A e^{i\theta}\right)t^{4} +O(t^{5})\right)\cdot\] \[\cdot \left(i+ \left(-\frac{i}{2} A e^{-2i\theta}+\frac{3i}{2}\overline  A e^{2i\theta}\right) t^4  + O(t^5)\right)d\theta =0 \ \ \mbox{ for all } j\geq 1,\]
and after carrying out the products and dividing by $i$,
\[\int_0^{2\pi}   c(\theta,t)\left(e^{ij\theta}+\frac{j-4}{2}\left( A e^{i(j-2)\theta}+\overline  A e^{i(j+2)\theta}\right)t^{4} +O(t^{5})\right)d\theta =0 \]
 for all  $j\geq 1$.
Once again we differentiate under the integral sign with respect to $t$ four times:
\begin{equation*}\begin{aligned}
\sum_{\ell=0}^4\binom{4}{\ell}\int_0^{2\pi} &\sum_{k=-\infty}^{+\infty}  \frac{d^\ell \gamma_k}{dt^\ell}(t)e^{ik\theta}\cdot \\ & \cdot\left(\delta_{4}^\ell e^{ij\theta} + \frac{4!}{\ell!}\frac{j-4}{2} \left( A e^{i(j-2)\theta} + \overline  A e^{i(j+2)\theta}\right) t^{\ell} + O(t^{\ell+1})\right) d\theta =0 \end{aligned}
\end{equation*}
which translates into
\begin{equation*}
\sum_{\ell=0}^4\binom{4}{\ell}\left(\frac{d^\ell \overline \gamma_j}{dt^\ell}(t)\delta_{4}^\ell + \frac{4!}{\ell!}\frac{j-4}{2}\left( A \frac{d^\ell \overline \gamma_{j-2}}{dt^\ell}(t) + \overline  A \frac{d^\ell \overline \gamma_{j+2}}{dt^\ell}(t)\right) t^{\ell} + O(t^{\ell+1})\right) =0.
\end{equation*}
Taking $j=2$ we have
\begin{equation*}
\sum_{\ell=0}^4\binom{4}{\ell}\left(\frac{d^\ell \overline \gamma_2}{dt^\ell}(t)\delta_{4}^\ell - \frac{4!}{\ell!}\left( A \frac{d^\ell \overline \gamma_{0}}{dt^\ell}(t) + \overline  A \frac{d^\ell \overline \gamma_{4}}{dt^\ell}(t)\right) t^{\ell} + O(t^{\ell+1})\right) =0;
\end{equation*}
except for $\ell=4$ and $\ell=0$ every term is $O(t)$, hence we get
\begin{equation*}
\frac{d^4 \overline \gamma_{2}}{dt^4}(t) - 4! \left( A \overline \gamma_{0}(t) + \overline  A \overline \gamma_{4}(t)\right) = O(t).
\end{equation*}
Recalling that $\gamma_0\equiv 1$ due to our normalization, and using (\ref{eq:order4}), (\ref{eq:gamma2}), we deduce that $ A=O(t)$. This is only possible if  $ A=0$.

\bibliographystyle{abbrv}
\bibliography{segre}
\end{document}